\documentclass[12pt]{amsart}
\usepackage{amscd}
\usepackage{amsmath}
\usepackage{verbatim}

\usepackage[cmtip,arrow]{xy}

\usepackage{pb-diagram,pb-xy}

\usepackage{amssymb}

\newtheorem{thm}[equation]{Theorem}
\newtheorem{prop}[equation]{Proposition}

\newtheorem{lemma}[equation]{Lemma}
\newtheorem{cor}[equation]{Corollary}

\theoremstyle{definition}
\newtheorem{rem}[equation]{Remark}
\newtheorem{example}[equation]{Example}

\newcommand{\Steen}{\operatorname{Sq}}

\newcommand{\rdim}{\operatorname{rdim}}

\renewcommand{\Im}{\mathop{\mathrm{Im}}}

\newcommand{\ind}{\mathop{\mathrm{ind}}}

\newcommand{\rk}{\mathop{\mathrm{rk}}}
\newcommand{\CH}{\mathop{\mathrm{CH}}\nolimits}

\newcommand{\SO}{\operatorname{\mathrm{SO}}}
\newcommand{\OO}{\operatorname{\mathrm{O}}}

\newcommand{\Ch}{\mathop{\mathrm{Ch}}\nolimits}

\newcommand{\res}{\mathop{\mathrm{res}}\nolimits}

\newcommand{\pr}{\operatorname{\mathit{pr}}}

\newcommand{\mult}{\operatorname{mult}}

\newcommand{\id}{\mathrm{id}}

\newcommand{\Z}{\mathbb{Z}}
\newcommand{\F}{\mathbb{F}}

\newcommand{\End}{\operatorname{End}}

\newcommand{\Aut}{\operatorname{Aut}}

\newcommand{\X}{\mathfrak{X}}
\newcommand{\Y}{\mathcal{Y}}

\newcommand{\pt}{\mathbf{pt}}

\newcommand{\disc}{\operatorname{disc}}

\newcommand{\compose}{\circ}

\renewcommand{\phi}{\varphi}

\newcommand{\mf}{\mathfrak}

\usepackage[hypertex]{hyperref}

\title
[Isotropy of involutions]
{Isotropy of orthogonal involutions}


\author
[N. Karpenko with appendix by J.-P.Tignol]
{Nikita A. Karpenko\\
with an Appendix by Jean-Pierre Tignol}

\address
{Nikita A. Karpenko\\
UPMC Univ Paris 06\\
Institut de Math\'ematiques de Jussieu\\
F-75252 Paris\\
FRANCE}

\address
{{\it Web page:}
{\tt www.math.jussieu.fr/\~{ }karpenko}}

\email {karpenko {\it at} math.jussieu.fr}

\address
{Jean-Pierre Tignol,
Zukunftskolleg, Universit\"at Konstanz, D-78457 Konstanz,
Germany, and ICTEAM, Universit\'e
Catholique de Louvain, B-1348 Louvain-la-Neuve, Belgium}

\date
{January, 2010. Updated: December 2010}

\thanks
{The statement of Proposition \ref{maksim} has been tested
and Example \ref{example} has been detected
using the Maple {\em Chow Ring Package}
by S. Nikolenko, V. Petrov, N. Semenov, and K. Zainoulline.}

\begin{document}

\begin{abstract}
An orthogonal involution on a central simple algebra becoming isotropic over any splitting field of the algebra, becomes isotropic over a finite odd degree extension of the base field (provided that the characteristic of the base field is not $2$).
The proof makes use of a structure theorem for Chow motives with finite coefficients of projective homogeneous varieties, of incompressibility of
certain generalized Severi-Brauer varieties, and of Steenrod operations.
\end{abstract}

\maketitle




The main result of this paper is as follows:

\begin{thm}
\label{main}
Let $F$ be a field of characteristic not $2$,
$A$ a central simple $F$-algebra, $\sigma$ an orthogonal involution on $A$.
The following two conditions are equivalent:
\begin{enumerate}
\item
$\sigma$ becomes isotropic over
any splitting field of $A$;
\item
$\sigma$ becomes isotropic over some finite odd degree extension of the base field.
\end{enumerate}
\end{thm}

The proof of Theorem \ref{main} is given in the very end of the paper;
it makes use of Chow motives with finite coefficients, of incompressibility of certain projective homogeneous varieties,
and of Steenrod operations.
A sketch of the proof is given shortly below.

For $F$ with no finite field extensions of odd degree, Theorem \ref{main}
proves \cite[Conjecture 5.2]{MR1751923}.
(I learned this conjecture in 1994 from A. Wadsworth during a Luminy conference.)
For general $F$, the question whether condition $(2)$ implies isotropy of $\sigma$ over $F$ remains open.
Note that any orthogonal involution becomes isotropic over some $2$-primary field extension, so that the mentioned open question
is about existence of a rational point on a variety possessing a $0$-cycle of degree $1$.
Such a question can hardly be attacked by the methods of the paper, so that Theorem \ref{main} seems to be the best possible result
in this direction which can be achieved by such methods.

\bigskip
The general reference on central simple algebras and involutions is \cite{MR1632779}.

The implication $(2)\Rightarrow(1)$ is a consequence of the Springer theorem on quadratic forms.
We only prove the implication $(1)\Rightarrow(2)$.
Condition $(1)$ is equivalent to the condition that $\sigma$ becomes isotropic over some (and therefore any) generic
splitting field of the algebra, such as the function field of the Severi-Brauer variety of any central simple
algebra Brauer-equivalent to $A$.

We prove this theorem over all fields simultaneously using an induction on the index $\ind A$ of $A$.
The case of $\ind A=1$ is trivial.
The case of $\ind A=2$ is done in \cite{MR1850658} (with ``$\sigma$ is isotropic (over $F$)'' in place of condition (2)).
From now on we are assuming that $\ind A>2$.
Therefore $\ind A=2^r$ for some integer $r\geq2$.

Let us list our basic notation:
$F$ is a field of characteristic different from $2$;
$r$ is an integer $\geq2$;
$A$ is a central simple $F$-algebra of the index $2^r$;
$\sigma$ is an orthogonal involution on $A$;
$D$ is a central division $F$-algebra (of degree $2^r$) Brauer-equivalent to $A$;
$V$ is a right $D$-module  with an isomorphism $\End_D(V)\simeq A$;
$v$ is the $D$-dimension of $V$
(therefore $\rdim V=\deg A=2^r\cdot v$, where $\rdim V:=\dim_FV/\deg D$ is the reduced dimension of $V$);
we fix an orthogonal involution $\tau$  on $D$;
$h$ is a hermitian (with respect to $\tau$) form on $V$ such that the involution $\sigma$ is adjoint to $h$;
$\X=X(2^r;(V,h))$ is the variety of totally isotropic submodules in $V$ of the reduced dimension
$2^r$ which is isomorphic (via Morita equivalence) to the variety $X(2^r;(A,\sigma))$ of right totally isotropic ideals in $A$ of the same
reduced dimension;
$\Y=X(2^{r-1};D)$ is the variety of right ideals in $D$ of reduced dimension $2^{r-1}$.

We assume that the hermitian form $h$ (and therefore, the involution $\sigma$)
becomes isotropic over the function field of the Severi-Brauer variety $X(1;D)$ of $D$,
and we want to show that $h$ (and $\sigma$) becomes isotropic over a finite odd degree extension of $F$.
By \cite{isotropy}, the Witt index of $h$ (which coincides with the Witt index of $\sigma$)
over this function field is at least $2^r=\ind A$.
In particular, $v\geq2$.
If the Witt index is bigger than $2^r$, we replace $V$ by a submodule in $V$ of $D$-codimension $1$
(that is, of the reduced dimension $2^r(v-1)$) and we replace $h$
by its restriction on this new $V$.
The Witt index of $h_{F(X(1;D))}$ drops by at most $2^r$ or stays unchanged.
We repeat the procedure until the Witt index becomes equal to $2^r$
(we come down eventually to the Witt index $2^r$ because the Witt index is at most $2^r$ for $V$ with
$\dim_DV=2$).

If $\dim_DV=2$, then $h$ becomes hyperbolic over $F\big(X(1;D)\big)$.
Therefore, by the main result of \cite{hypernew}, $h$ is hyperbolic over $F$ and we are done.
By this reason, we assume that $\dim_DV\geq3$, that is, $v\geq3$.
In particular, the variety $\X$ is projective {\em homogeneous}
(in the case of $v=2$, the variety $\X$ has two connected components each of which is homogeneous).

\begin{rem}
Note that the hyperbolicity theorem (HT) of \cite{hypernew} is a formal consequence of Theorem \ref{main}.
Keeping the case $v=2$ (therefore avoiding the only point where we use HT) and
slightly modifying the sequel, one can get a proof of Theorem \ref{main} which does not rely on HT (see Remark \ref{rem12}).
This will give a new proof of HT which (although having much in common) is (at several points) essentially different from the original one.
\end{rem}

The variety $\X$ has an $F\big(X(1;D)\big)$-point and $\ind D_{F(\Y)}=2^{r-1}$.
Consequently, by the induction hypothesis, the variety
$\X_{F(\Y)}$ has an odd degree closed point.
We prove Theorem \ref{main} by showing that the variety $\X$ has an odd degree closed point.
Here is a sketch of the proof:

\begin{proof}[Sketch of Proof of Theorem \ref{main}]
We assume that the variety $\X$ (and therefore also $\X\times\X$) has no odd degree closed point and we are looking for a contradiction.
First we show that the Chow motive with coefficients in $\F_2$ of $\X$ contains a summand isomorphic to a shift of
the {\em upper} indecomposable summand $M_\Y$ of the motive of $\Y$ (Corollary \ref{upper r-1}),
where {\em upper} means that the $0$-codimensional Chow group of $M_\Y$ is non-zero.
(At this point we use the $2$-incompressibility of $\Y$ which is due to
\cite{MR2713320}.)
Moreover, the corresponding projector on $\X$ can be {\em symmetrized} (Proposition \ref{symsym}).
This makes it possible to compute the degree modulo $4$ of any integral representative of
the $0$-cycle class on $\X\times\X$, given by the value of the appropriate  Steenrod operation on this projector.
Namely (see Corollary \ref{corsim}), this degree is identified with the {\em rank} of $M_\Y$ and therefore is $2$ modulo $4$ by a result of
\cite{upper} (which is a consequence of the $2$-incompressibility of $\Y$ and a structure theorem for motives with finite coefficients of
projective homogeneous varieties established in \cite{upper} and generalized in \cite{outer};
this structure theorem says that any indecomposable summand of the motive of a projective $G$-homogeneous variety $X$,
where $G$ is semisimple affine algebraic group,
is isomorphic to the {\em upper} indecomposable summand of another projective $G$-homogeneous
variety $X'$ such that the Tits index of $G_{F(X')}$ contains the Tits index of $G_{F(X)}$).
On the other hand, a computation of  Steenrod operations on the split orthogonal grassmannian $\bar{\X}$
(Proposition \ref{gras}) allows one to show that
the above degree is $0$ modulo $4$.
This is the required contradiction.
\end{proof}

We need an enhanced version of \cite[Proposition 4.6]{hypernew}.
This is a statement about the {\em Grothendieck Chow motives} (see \cite[Chapter XII]{EKM}) with coefficients in a prime field $\F_p$
(which we shall apply to $p=2$).
We write $\Ch$ for Chow groups with coefficients in $\F_p$ and we write
$M(X)$ for the motive of a complete smooth $F$-variety $X$.
Saying ``sum of motives'', we always mean the direct sum.
We call $X$ {\em split}, if $M(X)$ is isomorphic to a sum of Tate motives (which are defined as shifts of the motive of a point),
and we call $X$ {\em geometrically split}, it it splits over an extension of the base field.
We say that $X$ satisfies {\em nilpotence principle}, if for any field extension $E/F$ the kernel of the change of field homomorphism
$\End M(X)\to \End M(X)_E$ consists of nilpotents.
Finally, $X$ is {\em $p$-incompressible}, if it is connected and for any proper closed subvariety $Y\subset X$, the degree of any closed point
on $Y_{F(X)}$ is divisible by $p$.

The base field $F$ may have arbitrary characteristic in this statement:

\begin{prop}
\label{prop}
Let $Y$ be a geometrically split,
geometrically irreducible $F$-variety satisfying the nilpotence principle and let
$X$ be a smooth complete $F$-variety.
Assume that there exists a field extension $E/F$ such that
\begin{enumerate}
\item
for some field extension $\overline{E(Y)}/E(Y)$,
the image of the change of field homomorphism $\Ch(X_{E(Y)})\to\Ch(X_{\overline{E(Y)}})$
coincides with the image of the change of field homomorphism
$\Ch(X_{F(Y)})\to\Ch(X_{\overline{E(Y)}})$;
\item
the $E$-variety $Y_E$ is $p$-incompressible;
\item
a shift of the upper indecomposable summand of $M(Y)_E$ is a summand of $M(X)_E$.
\end{enumerate}
Then the same shift of the upper indecomposable summand of $M(Y)$ is a summand of $M(X)$.
\end{prop}

\begin{proof}
We recall that this Proposition is an enhanced version of \cite[Proposition 4.6]{hypernew}.
The only difference with the original version is in the condition (1):
the field extension $E(Y)/F(Y)$ is assumed to be purely transcendental in the original version.
However, only the new condition (1), a consequence of the pure transcendentality, is used in the original proof.
\end{proof}

Everywhere below, the prime $p$ is $2$.
We are going to apply Proposition \ref{prop} (with $p=2$) to $Y=\Y$, $X=\X$, and $E=F(\X)$.
We do not know if
the field extension $E(\Y)/F(\Y)$ is purely transcendental because
we do not know whether the variety $\X_{F(\Y)}$ has a rational point
(we only know that this variety has an odd degree closed point).

Next we are going to check that conditions $(1)$--$(3)$ of Proposition \ref{prop} are satisfied for these $Y,X,E$.
We start with condition $(3)$ for which
we need a motivic decomposition of $X_E=\X_{F(\X)}$.
We have the decomposition of \cite{MR2110630} arising from the fact that
$\X(F(\X))\ne\emptyset$.
More generally, the ``same'' decomposition holds for $F(\X)$ replaced by any field $K/F$ with $\X(K)\ne\emptyset$.
Over such $K$, the hermitian form $h$ decomposes in the orthogonal sum of a hyperbolic
$D_K$-plane and a hermitian form $h'$ on a right $D_K$-module $V'$ with $\rdim V'=2^r(v-2)$.

It requires some work to derive the decomposition from the general theorem of \cite{MR2110630}.
We use a ready answer from \cite{MR1758562}, where the projective homogeneous varieties under the
{\em classical} semisimple affine algebraic groups have been treated:

\begin{lemma}[{\cite[Corollary 15.14]{MR1758562}}]
\label{X over F(X)}
$M(\X_K)\simeq$
$$
\bigoplus_{i,j}M\big(X(i,i+j;D_K)\times X(j;(V',h'))\big)\big(i(i-1)/2+j(i+j)+i(\rdim V'-j)\big),
$$
where $X(i,i+j;D_K)$ is the variety of flags given by a right ideal in the $K$-algebra $D_K$
of the reduced dimension $i$ contained
in a right ideal of the reduced dimension $i+j$
(this is a non-empty variety if and only if $0\leq i\leq i+j\leq\deg D$).
\end{lemma}

\begin{proof}
Unfortunately, \cite[Corollary 15.14]{MR1758562} is not the above statement on
motives, but its consequence. However, the needed statement on motives is
actually proved in the proof of \cite[Corollary 15.14]{MR1758562}.
\end{proof}

In particular, a shift of the motive of the variety $\Y_{F(\X)}$ is a motivic summand of $\X_{F(\X)}$:
namely, the summand of Lemma \ref{X over F(X)} given by $i=2^{r-1}$ and $j=0$ (with $K=F(\X)$).
This summand has as the shifting number the integer
\begin{equation}
\label{def-n}
n:=2^{r-2}(2^{r-1}-1)+2^{2r-1}(v-2).
\end{equation}
We note that $\dim\X=2^{r-1}(2^r-1)+2^{2r}(v-2)$, $\dim\Y=2^{2r-2}$, and therefore
$$
n=(\dim\X-\dim\Y)/2.
$$

We have checked condition $(3)$ of Proposition \ref{prop} and we start checking condition $(2)$.
By
\cite{MR2713320}
(see \cite{upper} for a different proof and generalizations),
the variety $\Y_{F(\X)}$ is $2$-incompressible if (and only if) the division algebra $D$ remains division over
the field $F(\X)$.
This is  indeed the case:

\begin{lemma}
\label{index reduction}
The algebra $D_{F(X)}$ is division, that is,
$\ind D_{F(\X)}=\ind D$.
\end{lemma}

\begin{proof}
Of course, the statement can be checked using the index reduction formulas of \cite{MR1415325}
(in the inner case, that is, in the case when the discriminant of $h$ is trivial) and of
\cite{MR1628279} (in the outer case).
However, we prefer to do it in a different way which is more internal with respect to the methods of this paper.

Assume that $\ind D_{F(\X)}<\ind D$.
Then $\Y(F(\X))\ne\emptyset$.
Since in the same time the variety $\X_{F(\Y)}$ has an odd degree closed point,
it follows (by the main property of the upper motives established in \cite[Corollary 2.15]{upper})
that the upper indecomposable motivic summand of $\Y$ is a motivic summand of $\X$.
This implies (because the variety $\Y$ is $2$-incompressible)
that the complete motivic decomposition of the variety $\X_{F(\Y)}$
contains the Tate summand $\F_2(\dim\Y)=\F_2(2^{2r-2})$.
On the other hand, all the summands of the motivic decomposition of Lemma \ref{X over F(X)} (applied to the field
$K=F(\Y)$) are shifts of the motives of anisotropic varieties besides the following three:
$\F_2$ (given by $i=j=0$), $\F_2(\dim\X)=\F_2\big(2^{r-1}(2^r-1)+2^{2r}(v-2)\big)$ (given by $i=2^r$ and $j=0$), and
$M(\Y_{F(\Y)})(n)$
(given by $i=2^{r-1}$ and $j=0$) with $n$ defined in (\ref{def-n}).
Here a variety is called anisotropic, if all its closed points are of even degree.
The motive of an anisotropic variety does not contain Tate summands by \cite[Lemma 2.21]{upper}.
Taking into account the Krull-Schmidt principle of \cite{MR2264459} (see also \cite[\S2]{outer}),
we get a contradiction because
$0<2^{2r-2}< n$ (the assumption $v\geq3$ is used here).
\end{proof}

We have checked condition $(2)$ of Proposition \ref{prop}.
It remains to check condition $(1)$.

\begin{lemma}
\label{odd}
Let $L/K$ be a finite odd degree field extension of a field $K$ containing $F$.
Let $\bar{L}$ be an algebraically closed field containing $L$.
Then
$$
\Im\big(\res_{\bar{L}/L}:\CH(\X_L)\to\CH(\X_{\bar{L}})\big)=
\Im\big(\res_{\bar{L}/K}:\CH(\X_K)\to\CH(\X_{\bar{L}})\big).
$$
\end{lemma}

\begin{proof}
We write $I_L$ and $I_K$ for these images and we evidently have $I_K\subset I_L$.

Inside of $\bar{L}$, the variety $\X_K$ has
a finite $2$-primary splitting field $K'/K$.

If the discriminant $\disc h_K$ is trivial, then $[L:K]\cdot I_L\subset I_K$.
Since moreover $[K':K]\cdot\CH(\X_{\bar{L}})\subset I_K$ and $[K':K]$ is coprime with
$[L:K]$, it follows that $I_L\subset I_K$.

If $\disc h_K$ is non-trivial, then also $\disc h_L\ne1$ and
the group $G:=\Aut(\bar{L}/K)$, acting on $\CH(\X_{\bar{L}})$, acts trivially on $I_L$.
Therefore we still have $[L:K]\cdot I_L\subset I_K$.
Besides, $[K':K]\cdot\CH(\X_{\bar{L}})^G\subset I_K$, and it follows that $I_L\subset I_K$.
\end{proof}

We write $M_{\Y}$ for the upper indecomposable motivic summand of $\Y$.

\begin{cor}
\label{upper r-1}
$M_{\Y}(n)$ is a motivic summand of $\X$.
\end{cor}

\begin{proof}
As planned,
we apply Proposition \ref{prop} to $p=2$, $Y=\Y$, $X=\X$, and $E=F(\X)$.
There exists a finite odd degree extension $L/F(Y)$ such that $X(L)\ne\emptyset$.
The field extension $L(X)/L$ is purely transcendental.
Since $E(Y)\subset L(X)$, condition $(1)$ is satisfied
(with $\overline{E(Y)}$ being an algebraically closed field containing $L(X)$) by Lemma \ref{odd}.

As already pointed out,
condition $(2)$ is satisfied by Lemma \ref{index reduction}, and
condition $(3)$ is satisfied by Lemma \ref{X over F(X)}.
\end{proof}

We need the following enhancement of Corollary \ref{upper r-1}:

\begin{prop}
\label{symsym}
There exists a \underline{{\em symmetric}} projector $\pi_\X$ on $\X$ such that the motive
$(\X,\pi_\X)$ is isomorphic to $M_\Y(n)$.
\end{prop}

\begin{rem}
In fact, for {\em any} projector $\pi_\X$ on $\X$ such that the motive
$(\X,\pi_\X)$ is isomorphic to $M_\Y(n)$, the motive $(\X,\pi_\X^t)$ given by the transposition $\pi_\X^t$
of $\pi_\X$ is {\em isomorphic} to $(\X,\pi_\X)$.
However, $\pi_\X$ is not necessarily symmetric, that is, the equality $\pi_\X^t=\pi_\X$ may fail.
\end{rem}

\begin{proof}[Proof of Proposition \ref{symsym}]
Let us start by checking that the motive $M_\Y$ can be given by a symmetric projector $\pi_\Y$
on $\Y$.
The proof we give is valid for any projective homogeneous $2$-incompressible variety in place of the variety $\Y$.
Let $\pi$ be a projector on $\Y$ such that $(\Y,\pi)\simeq M_\Y$.
Since our Chow groups are with finite coefficients,
there exists an integer $l\geq1$ such that $\pi_\Y:=(\pi^t\compose\pi)^{\compose l}$ is a (symmetric) projector,
where $\pi^t$ is the transposition of $\pi$.
Since the variety $\Y$ is $2$-incompressible, $\mult\pi^t=1$ by \cite[\S2]{upper}, where $\mult$ is the multiplicity
(sometimes also called {\em degree} in the literature) of a correspondence.
It follows that $\mult\pi_\Y=1$
and therefore the motive $(\Y,\pi_\Y)$
is non-zero.
In the same time, it is a direct summand of the indecomposable
motive $(\Y,\pi)$ (the morphisms to and from $(\Y,\pi)$ having the identical composition are given, for instance, by
$\pi\compose\pi_\Y$ and simply $\pi_\Y$).
Therefore $M_\Y\simeq (\Y,\pi_\Y)$ by indecomposability of $(\Y,\pi)$.

Now let $\alpha:(\Y,\pi_\Y)(n)\to M(\X)$ and $\beta: M(\X)\to (\Y,\pi_\Y)(n)$ be morphisms with
$\beta\compose\alpha=\pi_\Y=\id_{(\Y,\pi_\Y)}$
(existing because $(\Y,\pi_\Y)(n)$ is a motivic summand of $\X$).
Note that $\alpha^t$ is a morphism
$$
M(\X)\to(\Y,\pi^t_\Y)(\dim\X-\dim\Y-n)=(\Y,\pi_\Y)(n)
$$
because $\pi^t_\Y=\pi_\Y$ and $2n=\dim\X-\dim\Y$.
There exists an integer $l\geq1$ such that $(\alpha^t\compose\alpha)^{\compose l}$ is a projector.
If  $\mult(\alpha^t\compose\alpha)\ne0$, then
$(\alpha^t\compose\alpha)^{\compose l}=\pi_\Y$.
Therefore $(\alpha\compose\alpha^t)^{\compose l}$ is a (symmetric) projector on $\X$ and
$\alpha:(\Y,\pi_\Y)(n)\to\big(\X,(\alpha\compose\alpha^t)^{\compose l}\big)$ is an isomorphism of motives,
so that we are done in this case.

Similarly, if $\mult(\beta\compose\beta^t)\ne0$, then $\beta^t:(\Y,\pi_\Y)(n)\to\big(\X,(\beta^t\compose\beta)^{\compose l}\big)$
for some (other) $l$ is an isomorphism, and we are done in this case also.

In the remaining case we have
$
\mult(\alpha^t\compose\alpha)=0=\mult(\beta\compose\beta^t)
$.
Let $\pt\in\Ch_0(\Y_{F(\Y)})$ be the class of a rational point.
The compositions $\alpha\compose([\Y_{F(\Y)}]\times\pt)\compose\beta$ and
$\beta^t\compose([\Y_{F(\Y)}]\times\pt)\compose\alpha^t$ are orthogonal projectors on $\X_{F(\Y)}$,
and each of two corresponding motives is isomorphic to $\F_2(n)$.
It follows that the complete motivic decomposition of $\X_{F(\Y)}$ contains two exemplars of $\F_2(n)$.
However, as shown in the end of the proof of Lemma \ref{index reduction},
the complete motivic decomposition of $\X_{F(\Y)}$ contains only one exemplar of $\F_2(n)$
(because the motive of $\Y_{F(\Y)}$ contains only one exemplar of $\F_2$).
\end{proof}

From now on we are assuming that all closed points on the variety $\X$ have even degrees.
Then all closed points on the product $\X\times\X$ also have even degrees.
Therefore the homomorphism $\deg\!/2:\Ch_0(\X\times\X)\to\F_2$ is defined
(as in \cite[\S5]{hypernew}).

\begin{cor}
\label{corsim}
Let $\pi_\X$ be as in Proposition \ref{symsym}.
Then $\pi_\X^2$ is a $0$-cycle class on $\X\times\X$ for which we have
$
(\deg\!/2)(\pi_\X^2)=1\in\F_2$.
\end{cor}

\begin{proof}
For any symmetric projector $\pi$ on $\X$, we have $(\deg\!/2)(\pi^2)=\rk(\X,\pi)/2\pmod{2}$,
where $\rk$ is the rank of the motive (the number of the Tate summands in the complete decomposition over a splitting field).
Indeed,
taking a complete motivic decomposition of $\bar{\X}$ (here and below $\bar{\X}$ is $\X$ over a splitting field of $\X$)
which is a refinement of the decomposition
$M(\X)\simeq(\X,\pi)\oplus(\X,\Delta_\X-\pi)$, we get
a homogeneous basis $B$ of $\Ch(\bar{\X})$  such that
$\bar{\pi}=\sum_{b\in B_\pi}b\times b^*$, where $B_\pi$ is a subset of $B$ and $\{b^*\}_{b\in B}$ is the dual basis.
Note that $\rk(\X,\pi)=\# B_\pi$.
For every $b\in B$, let us fix an integral representative $\mathfrak{b}\in\CH(\bar{\X})$ of $b$
and an integral representative $\mathfrak{b}^*\in\CH(\bar{\X})$ of $b^*$.
Then the sum $\sum_{b\in B_\pi}\mathfrak{b}\times \mathfrak{b}^*$,
as well as the sum $\sum_{b\in B_\pi}\mathfrak{b}^*\times \mathfrak{b}$,
is an integral representative of $\bar{\pi}$,
and for the integral degree homomorphism $\deg:\CH_0(\bar{\X})\to\Z$ we have:
\begin{equation*}
\deg\left(\Big(\sum_{b\in B_\pi}\mathfrak{b}\times \mathfrak{b}^*\Big)\Big(\sum_{b\in B_\pi}
\mathfrak{b}^*\times \mathfrak{b}\Big)\right)\equiv\#B_\pi\pmod{4}.
\end{equation*}

By definition of $\deg\!/2$, the element $(\deg\!/2)(\pi^2)\in\F_2$ is represented by the half of the degree of an arbitrary
integral representative of $\pi^2$ (over $F$ !).
So, let $\Pi\in\CH(\X\times\X)$ be an integral representative of $\pi$ ($\Pi$ does not need to be a projector).
Then $\Pi\cdot\Pi^t$ is a representative of $\pi^2$, so that
we have $(\deg\!/2)(\pi^2)=(\deg(\Pi\cdot\Pi^t))/2\pmod{2}$.
On the other hand,
there exists an element $\alpha\in\CH(\bar{\X}\times\bar{\X})$ such that
$$
\sum_{b\in B_\pi}\mathfrak{b}\times \mathfrak{b}^*=\bar{\Pi}+2\alpha,
$$
and we get the following congruences modulo $4$:
$$
\#B_\pi\equiv \deg\big((\bar{\Pi}+2\alpha)\cdot(\bar{\Pi}^t+2\alpha^t)\big)\equiv
\deg(\Pi\cdot\Pi^t)
$$
because $\deg(\alpha\cdot\bar{\Pi}^t)=\deg(\bar{\Pi}\cdot\alpha^t)$.

We have shown that $(\deg\!/2)(\pi_\X^2)=\rk(\X,\pi_\X)/2\pmod{2}$.
Now the rank of the motive $(\X,\pi_\X)$ coincides with
the rank of the motive $(\Y,\pi_\Y)\simeq M_\Y$ which is shown to be
$2$ modulo $4$ in \cite[Theorem 4.1]{upper}.
\end{proof}

The following Proposition is a general statement on the action of the cohomological Steenrod operation $\Steen^\bullet$
(see \cite[Chapter XI]{EKM})
on the Chow groups modulo $2$
of a split orthogonal grassmannian $G$ (which we shall apply to $G=\bar{\X}$, where, as above,  $\bar{\X}$ is $\X$ over a splitting field of $\X$):

\begin{prop}
\label{gras}
Let $d$ be an integer $\geq1$,
$m$ an integer satisfying $0\leq m\leq d-1$,
$G$ the variety of the totally isotropic ($m+1$)-dimensional subspaces of a hyperbolic
($2d+2$)-dimensional quadratic form $q$ (over a field of characteristic $\ne2$).
Then
for any integer $i> (d-m)(m+1)$ we have
$\Steen^i\Ch_i(G)=0$.
\end{prop}

\begin{proof}
Let $Q$ be the projective quadric of $q$,
$\Phi$ the variety of flags consisting of a line contained in a totally isotropic ($m+1$)-dimensional
subspace of $q$, and $\pr_G:\Phi\to G$, $\pr_Q:\Phi\to Q$ the projections.
We write $h\in\CH^1(Q)$ for the (integral) hyperplane section class and we write
$l_i\in\CH_i(Q)$, where $i=0,\dots,d$, for the (integral) class of an $i$-dimensional linear subspace in $Q$
(for $i=d$ we {\em choose} one of the two classes, call it $l_d$, and write $l'_d$ for the other).
As in \cite[\S2]{Vishik-u-invariant}, we define the integral classes
$$
W_i\in\CH^i(G)\;\; \text{ for }\;\; i=1,\dots, d-m
\;\;\text{ by }\;\; W_i:=(\pr_G)_*\pr_Q^*(h^{m+i})
$$
and we define the integral classes
$$
Z_i\in\CH^i(G)\;\;\text{ for }\;\;
i=d\!-\!m,\dots,2d\!-\!m\;\;\text{ by }\;\;
Z_i=(\pr_G)_*\pr_Q^*(l_{2d-m-i}).
$$

The elements $W_1,\dots,W_{d-m},Z_{d-m},\dots,Z_{2d-m}$ generate the ring $\CH(G)$ by
\cite[Proposition 2.9]{Vishik-u-invariant}.
We call them {\em the generators of} $\CH(G)$.
We refer to $W_1,\dots,W_{d-m}$ as $W$-generators,
and we refer to $Z_{d-m},\dots,Z_{2d-m}$ as $Z$-generators.

Note that $Z_{d-m}=(\pr_G)_*\pr_Q^*(l_d)$.
We also set $Z'_{d-m}=(\pr_G)_*\pr_Q^*(l'_d)$.
Since $l_d+l'_d=h^d$, we have $Z_{d-m}+Z'_{d-m}=W_{d-m}$.

Note that any element of $\OO_{2d+2}(F)\setminus \SO_{2d+2}(F)$ gives an automorphism of $G$ such that the corresponding
automorphism of the ring $\CH(G)$ acts trivially on
all the generators but $Z_{d-m}$ which is interchanged with $Z'_{d-m}$.

For any $i\geq0$, let $c_i\in\CH^i(G)$ be the $i$th Chern class of the quotient bundle on $G$.
According to \cite[Proposition 2.1]{Vishik-u-invariant}, $c_i=W_i$ for any $i$ for which $W_i$ is defined, and
$c_i=2Z_i$ for all $i$ satisfying $d-m<i\leq 2d-m$.

A computation similar to \cite[(86.15)]{EKM} (see also \cite[(44) and (45) in Theorem 3.2]{Kresch}) shows that
for any $i=d-m,\dots,2d-m$, the generators of $\CH(G)$ satisfy the following relation
$$
Z_i^2-Z_ic_i+Z_{i+1}c_{i-1}-Z_{i+2}c_{i-2}+\dots.
$$
(This is not and we do not need a complete list of relations.)

We denote the images of the generators of $\CH(G)$ under the epimorphism $\CH(G)\to\Ch(G)$ to the modulo $2$ Chow group using the
small letters $w$ and $z$ (with the same indices), and call them {\em the generators of} $\Ch(G)$.
We say that an element of $\Ch(G)$ is of level $l$, if it can be written as a sum of products of generators such that the number of the $z$-factors
in each product is at most $l$ (so, any level $l$ element is also of level $l+1$).
A $z$-generator raised to power $k$ is counted $k$ times here, that is, we are looking at the total degree assigning weight $1$
to each $z$-generator  (and weight $0$ to each $w$-generator).
For instance, the monomial $z_d^2$ is of level $2$
(but because of the relation $z_d^2=z_dc_d-z_{d+1}c_{d-1}+\dots$, the element $z_d^2$ is also of level $1$).

By \cite[Proposition 2.8]{Vishik-u-invariant}, the value of the total cohomological Steenrod operation
$\Steen^\bullet:\Ch(G)\to\Ch(G)$ on any single $z$-generator is of level $1$.
Similar computation shows that the value of $\Steen^\bullet$ on any $w$-generator is of level $0$.
Since $\Steen^\bullet$ is a ring homomorphism, it follows that for any $l\geq0$, the image under $\Steen^\bullet$ of a level $l$ element
is also of level $l$.

The above relations on the generators show that any element of $\Ch(G)$ is a polynomial of the generators such that
the exponent of any $z$-generator in any monomial of the polynomial is at most $1$.
Since the dimension of such (biggest-dimensional  level $m+1$) monomial $z_{d-m}\dots z_d$ is equal to
\begin{multline*}
\dim G-\Big((d-m)+\dots+d\Big)=\dim G - \Big((d-m)(m+1)+ m(m+1)/2\Big)=\\
(d-m)(m+1),
\end{multline*}
any homogeneous element $\alpha\in\Ch(G)$ of dimension $i>(d-m)(m+1)$ is of level $m$.
Therefore $\Steen^i(\alpha)\in\Ch_0(G)$ if also of level $m$.

We finish by showing that any level $m$ element in $\Ch_0(G)$ is $0$.
For this we turn back to the integral Chow group $\CH(G)$ and show that any odd degree element  $\beta\in\CH_0(G)$
is not of level $m$.
The integral version of the notion of level used here is defined in the same way as the above modulo $2$ version (using the generators of
$\CH(G)$ instead of the generators of $\Ch(G)$).

Since the description of the ring $\CH(G)$ does not depend on the base field $F$, we may assume that $G=G'_F$,
where $G'$  is the grassmannian of a {\em generic} quadratic form defined over a subfield $F'\subset F$.
We say that an element of $\CH(G)$ is {\em rational}, if it is in the image of the change of field homomorphism
$\res_{F/F'}:\CH(G')\to\CH(G)$.

For any $i\geq0$, the element $c_i$ is rational.
Therefore, for any $l\geq0$, the $2^l$-multiple of any level $l$ element in $\CH_0(G)$ is rational.
Indeed, this statement is a consequence of the formulas $W_i=c_i$ for any $i$ such that $W_i$ is defined,
and the formulas $Z_i+\sigma Z_i=c_i$ for any $i$ such that $Z_i$ is defined, where $\sigma$ is the ring automorphism of
$\CH(G)$ given by an element of $\OO_{2d+2}(F)\setminus\SO_{2d+2}(F)$ (note that $\sigma$ is the identity on $\CH_0(G)$).
The degree of any closed point on $G'$ is divisible by $2^{m+1}$.
Therefore the element $2^m\beta$ is not rational, and it follows that $\beta$ is not of level $m$.
\end{proof}

\begin{rem}
It might look strange that we are using the trick with the generic quadratic form proving a statement about a split quadratic form.
Indeed, the ring $\Ch(G)$ is completely described in terms of generators and relations  (moreover, the Pieri rule
\cite[Theorem 3.1]{Kresch} is obtained) and the action of the Steenrod operations is computed in terms of the generators.
However a direct proof based only on this information seems to be very complicated.
\end{rem}

\begin{rem}
\label{rem12}
The statement of Proposition \ref{gras} also holds in the case of $m=d$, that is, in the case of a split {\em maximal} orthogonal grassmannian.
The proof is even simpler and also the given proof of Proposition \ref{gras} can be easily modified to cover this case.
Using this, one can cover the case of $v=2$, excluded in the very beginning, and obtain this way a new proof for the hyperbolicity result
of \cite{hypernew}.
\end{rem}

\begin{cor}
\label{maksim}
For any integer $i\geq n$ (where $n$ is as in (\ref{def-n})) we have
$\Steen^i\Ch_i(\bar{\X})=0$.
\end{cor}

\begin{proof}
We apply Proposition \ref{gras} to $G=\bar{\X}$.
We have $d=2^{r-1}v-1$ and $m=2^r-1\leq d-1$ (because $v\geq3$).
Therefore $(d-m)(m+1)=2^{2r-1}(v-2)$ and
$$
n:=2^{r-2}(2^{r-1}-1)+2^{2r-1}(v-2)>(d-m)(m+1)
$$
(because $r\geq2$).
\end{proof}

\begin{example}
\label{example}
Corollary \ref{maksim} fails for $r=1$.
For instance, if $v=6$ (and therefore $d=5$), we have:
$n=8$, $z_4z_5\in\Ch^9(\bar{\X})=\Ch_8(\bar{\X})$, and $\Steen^8(z_4z_5)\ne0$.
Therefore, an additional argument is needed to prove the quaternion case
(fortunately already proved in \cite{MR1850658}) by the method of this paper.
\end{example}

\begin{proof}[Proof of Theorem \ref{main}]
We are going to show that $(\deg\!/2)(\pi_\X^2)=0$.
This will contradict to Corollary \ref{corsim} thus proving Theorem \ref{main}.

Since $\pi_\X^2=\Steen^{\dim\X}\pi_\X$, we have
$(\deg\!/2)(\pi_\X^2)=(\deg\!/2)(\Steen^\bullet\pi_\X)$.
Let $\alpha:M(\Y)(n)\to M(\X)$ and $\beta: M(\X)\to M(\Y)(n)$ be morphisms of motives with
$\alpha\compose\beta=\pi_\X$
and let
$$
\pr^{\X\Y\X}_{\X\X}:\X\times\Y\times\X\to\X\times\X
$$
be the projection.
Since $\alpha\compose\beta=(\pr^{\X\Y\X}_{\X\X})_*\big(([\X]\times\alpha)\cdot(\beta\times[\X])\big)$, we have
$$
\Steen^\bullet\pi_\X=
(\pr^{\X\Y\X}_{\X\X})_*\Big(\big([\X]\times\Steen^\bullet(\alpha)\big)\cdot\big(\Steen^\bullet(\beta)\times[\X]\big)\cdot
\big([\X]\times c_\bullet(-T_\Y)\times[\X]\big)\Big),
$$
where $T_\Y$ is the tangent bundle of $\Y$ and $c_\bullet$ is the total Chern class modulo $2$.
Let $\mf{a}$ and $\mf{b}$ be integral representatives of $\Steen^\bullet(\alpha)$ and
$\Steen^\bullet(\beta)$.
It suffices to show that the degree of the integral cycle class
$$
\mf{d}:=(\pr^{\X\Y\X}_{\X\X})_*\Big(\big([\X]\times\mf{a}\big)\cdot\big(\mf{b}\times[\X]\big)\cdot
\big([\X]\times \mf{c}_\bullet(-T_\Y)\times[\X]\big)\Big)
$$
is divisible by $4$, where $\mf{c}_\bullet$ stands for the {\em integral} total Chern class.

We have
\begin{multline*}
(\pr^{\X\Y\X}_{\Y\X})_*\Big(\big([\X]\times\mf{a}\big)\cdot\big(\mf{b}\times[\X]\big)\cdot
\big([\X]\times \mf{c}_\bullet(-T_\Y)\times[\X]\big)\Big)=\\
\mf{a}\cdot\big((\pr^{\X\Y}_\Y)_*(\mf{b})\times[\X]\big)\cdot \big(\mf{c}_\bullet(-T_\Y)\times[\X]\big)
\end{multline*}
and
\begin{multline*}
(\pr^{\Y\X}_\Y)_*\Big(
\mf{a}\cdot\big((\pr^{\X\Y}_\Y)_*(\mf{b})\times[\X]\big)\cdot \big(\mf{c}_\bullet(-T_\Y)\times[\X]\big)
\Big)=\\
(\pr^{\Y\X}_\Y)_*(\mf{a})\cdot(\pr^{\X\Y}_\Y)_*(\mf{b})\cdot \mf{c}_\bullet(-T_\Y).
\end{multline*}
Therefore
$$
\deg(\mf{d})=\deg\Big((\pr^{\Y\X}_\Y)_*(\mf{a})\cdot(\pr^{\X\Y}_\Y)_*(\mf{b})\cdot \mf{c}_\bullet(-T_\Y)\Big)
$$
and it suffices to show that the cycle classes $(\pr^{\Y\X}_\Y)_*(\bar{\mf{a}})$ and
$(\pr^{\X\Y}_\Y)_*(\bar{\mf{b}})$ are divisible by $2$.

The (modulo $2$) cycle class $\bar{\alpha}$ is a sum of $a'\times a$ with
some $a'\in\Ch(\bar{\Y})$ and some
homogeneous $a\in\Ch(\bar{\X})$ of dimension
$\geq n$.
By Corollary \ref{maksim}, $\deg\Steen^\bullet(a)=0\in\F_2$ for such $a$.
Therefore
$(\pr^{\Y\X}_\Y)_*\big(\Steen^\bullet(\bar{\alpha})\big)=0$ and
the integral cycle class $(\pr^{\Y\X}_\Y)_*(\bar{\mf{a}})$,
which represents the modulo $2$ cycle class $(\pr^{\Y\X}_\Y)_*\big(\Steen^\bullet(\bar{\alpha})\big)$,
is divisible by $2$.
Similarly,
the cycle class $\bar{\beta}$ is a sum of $b\times b'$ with
some $b'\in\Ch(\bar{\Y})$ and some
homogeneous $b\in\Ch(\bar{\X})$ of dimension
$\geq n$, and it follows that
the cycle class $(\pr^{\X\Y}_\Y)_*(\bar{\mf{b}})$ is also divisible by $2$.
\end{proof}


\section*{Appendix: Isotropy of symplectic and unitary involutions}
\begin{center}
by Jean-Pierre Tignol\footnote{This appendix was written while the
  author was a Senior Fellow of the Zukunftskolleg, Universit\"at
  Konstanz, whose hospitality is gratefully acknowledged.}
\end{center}
\bigbreak
\par
\setcounter{section}{0}
\renewcommand\thesection{\Alph{section}}
\newtheorem{thmapp}[section]{Theorem}

Using a technique from \cite{Tignol-appendix}, we derive from
Theorem~\ref{main} the following analogues for symplectic and unitary
involutions:

\begin{thmapp}
  \label{symplectic}
  Let $A$ be a central simple algebra over a field $F$ of
  characteristic different from~$2$ and let $\sigma$ be a symplectic
  involution on $A$. The following conditions are equivalent:
  \begin{enumerate}
  \item
  $\sigma$ becomes isotropic over every field extension $E$ of $F$
  such that $\ind A_E=2$;
  \item
  $\sigma$ becomes isotropic over some odd-degree field extension of
  $F$.
  \end{enumerate}
\end{thmapp}

If $A$ is split, then (1) is void and (2) always holds since
symplectic involutions on split algebras are adjoint to alternating
forms.

\begin{thmapp}
  \label{unit}
  Let $B$ be a central simple algebra of exponent~$2$ over a field $K$
  of characteristic different from~$2$, let $\tau$ be a unitary
  involution on $B$, and let $F\subset K$ be the subfield of $K$ fixed
  under $\tau$. The following conditions are equivalent:
  \begin{enumerate}
  \item
  $\tau$ becomes isotropic over every field extension $E$ of $F$ such
  that $B\otimes_FE$ is split;
  \item
  $\tau$ becomes isotropic over some odd-degree field extension of $F$.
  \end{enumerate}
\end{thmapp}

In (1) it suffices to consider field extensions $E$ that are linearly
disjoint from $K$, since $\tau$ becomes isotropic over every field
extension containing $K$.

In each case, (2)~$\Rightarrow$~(1) readily follows from Springer's
theorem on the anisotropy of quadratic forms under odd-degree
extensions: symplectic involutions on central simple algebras of
index~$2$ are adjoint to hermitian forms $h$ over quaternion algebras,
which are isotropic if and only if the associated quadratic form
$h(x,x)$ is isotropic. Likewise, unitary involutions on split central
simple algebras are adjoint to hermitian forms over quadratic
extensions, and the same observation applies. Therefore, we just prove
(1)~$\Rightarrow$~(2).

\begin{proof}[Proof of Theorem~\ref{symplectic}]
  Since (2) holds when $A$ is split, we may assume $A$ is not
  split. Adjoining to $F$ two Laurent series indeterminates $x$, $y$,
  we let $\widehat F=F((x))((y))$ and consider the quaternion algebra
  $(x,y)_{\widehat F}$ with its conjugation involution $\gamma$. Let
  $\widetilde A= A\otimes_F(x,y)_{\widehat F}$ with the involution
  $\widetilde\sigma=\sigma\otimes\gamma$, which is orthogonal since
  $\sigma$ is symplectic. Let $E$ be the function field of the
  Severi--Brauer variety of $\widetilde A$. The Brauer group kernel of
  the scalar extension map from $\widehat F$ to $E$ is generated by
  $\widetilde A$, hence $A$ is not split over $E$; but $A_E$ is
  Brauer-equivalent to $(x,y)_E$, hence $\ind A_E=2$. Assuming~(1), we
  have $\sigma$ isotropic over $E$, hence $\widetilde\sigma$ also
  becomes isotropic over $E$ and therefore, by Theorem~\ref{main},
  there is an odd-degree extension $L$ of $\widehat F$ over which
  $\widetilde\sigma$ becomes isotropic. The $x,y$-adic valuation on
  $\widehat F$ (with value group $\mathbb{Z}^2$ ordered
  lexicographically from right to left) is Henselian, hence it extends
  to a valuation $v$ on $L$. To prove~(2), we show that $\sigma$
  becomes isotropic over the residue field $\overline{L}$, which is an
  odd-degree extension of $F$.

  Let $\Gamma=v(L^\times)\subset\mathbb{Q}^2$. Since $L$ is an
  odd-degree extension of $\widehat F$, the quaternion algebra
  $(x,y)_{\widehat F}$ remains a division algebra over $L$, hence $v$
  extends to a valuation on $(x,y)_L$ defined by
  \[
  v(q)=\textstyle\frac12 v\bigl(\operatorname{Nrd}(q)\bigr) \in
  \frac12\Gamma \cup\{\infty\} \qquad\text{for $q\in(x,y)_L$,}
  \]
  where $\operatorname{Nrd}$ is the reduced norm. Since
  $(\Gamma{:}\mathbb{Z}^2)$ is odd, we have $v(x)$, $v(y)\notin
  2\Gamma$, and the residue division algebra $\overline{(x,y)_L}$ is
  therefore easily checked to be $\overline{L}$. We further extend $v$
  to a map
  \[
  w\colon \widetilde A_L\to \textstyle\frac12\Gamma\cup\{\infty\}
  \]
  as follows: let $(a_i)_{i=1}^n$ be a base of $A$, so every element
  in $\widetilde A_L$ has a unique representation of the form $\sum_i
  a_i\otimes q_i$ for some $q_i\in(x,y)_L$; we set
  \[
  w(\textstyle\sum_i a_i\otimes q_i) = \min\{v(q_i)\mid 1\leq i\leq
  n\}.
  \]
  The map $w$ is not a valuation\footnote{The map $w$ is a
    \emph{gauge} in the terminology of [J.-P. Tignol, A.R. Wadsworth,
    Value functions and associated graded rings for semisimple
    algebras, \textit{Trans. Amer. Math. Soc.} \textbf{362} (2010)
    687--726].}
  ($\widetilde A_L$ is not a division algebra) but it satisfies
  $w(a+b) \geq \min\{w(a), w(b)\}$ and $w(ab)\geq w(a)+w(b)$ for $a$,
  $b\in \widetilde A_L$. (It is also easy to see that it does not
  depend on the choice of the base $(a_i)_{i=1}^n$.) We may therefore
  consider a residue algebra $(\widetilde A_L)_0$, which consists of
  residue classes of elements $a$ such that $w(a)\geq0$ modulo
  elements $a$ such that $w(a)>0$. Since
  $\overline{(x,y)_L}=\overline{L}$, we have $(\widetilde
  A_L)_0=A_{\overline{L}}$. Moreover,
  $w\bigl(\widetilde\sigma(a)\bigr) = w(a)$ for all $a\in\widetilde
  A_L$, and if $w(a)=0$ we have
  \[
  \overline{\widetilde\sigma(a)}=\sigma_{\overline{L}}(\overline{a}).
  \]
  Now, since $\widetilde\sigma_L$ is isotropic we may find a nonzero
  $a\in \widetilde A_L$ such that $\widetilde\sigma_L(a)\cdot
  a=0$. Multiplying $a$ on the right by a suitable quaternion in
  $(x,y)_L$, we may assume $w(a)=0$, hence $\overline{a}\in
  A_{\overline{L}}$ is defined and nonzero. We have
  \[
  \sigma_{\overline{L}}(\overline{a})\cdot \overline{a} =
  \overline{\widetilde\sigma_L(a)\cdot a} =0,
  \]
  hence $\sigma_{\overline{L}}$ is isotropic, as claimed.
\end{proof}

\begin{proof}[Proof of Theorem~\ref{unit}]
  As in \cite[A.2]{Tignol-appendix}, we choose an orthogonal involution $\nu$
  on $B$ and set $g=\nu\circ\tau$, which is an outer automorphism of
  $B$. Consider the algebra of twisted Laurent series $\widetilde
  B=B((\xi;g))$ in one indeterminate $\xi$. It carries an orthogonal
  involution $\widetilde\tau$ extending $\tau$ such that
  $\widetilde\tau(\xi)=\xi$. Let $u\in B^\times$ be such that
  $\nu(u)=\tau(u)=u$ and $g^2(b)=ubu^{-1}$ for all $b\in B$. The
  center of $\widetilde B$ is $F((x))$ where $x=u^{-1}\xi^2$. Let $E$
  be the function field of the Severi--Brauer variety of $\widetilde
  B$. Extension of scalars to $E$ splits $B$, since $B\otimes_FF((x))$
  is the centralizer of $K$ in $\widetilde B$. Therefore,
  assuming~(1), $\tau$ becomes isotropic over $E$, hence
  $\widetilde\tau$ also becomes isotropic over $E$. By
  Theorem~\ref{main}, it follows that there is an odd-degree extension
  $L$ of $F((x))$ over which $\widetilde\tau$ becomes isotropic. The
  $x$-adic valuation on $F((x))$ (with value group $\mathbb{Z}$)
  extends to a valuation $v$ on $L$ with value group $\Gamma\subset
  \mathbb{Q}$, and further to a map
  \[
  w\colon\widetilde B_L\to\textstyle\frac12\Gamma\cup\{\infty\}
  \]
  defined as follows: let $(b_i)_{i=1}^n$ be an $F$-base of $B$, so
  every element in $\widetilde B_L$ has a unique representation of the
  form $\sum_ib_i\otimes\ell_i +(\sum_j b_j\otimes\ell'_j)\xi$ for
  some $\ell_i$, $\ell'_j\in L$; set
  \[
  w(\textstyle\sum_ib_i\otimes\ell_i+(\sum_j b_j\otimes\ell'_j)\xi) =
  \min\{v(\ell_i), v(\ell'_j)+\frac12\mid 1\leq i,j\leq n\}.
  \]
  The corresponding residue ring $(\widetilde B_L)_0$ is
  $B\otimes_F\overline{L}$, and for $b\in \widetilde B_L$ with
  $w(b)=0$ we have $w\bigl(\widetilde\tau_L(b)\bigr)=0$ and
  $\overline{\widetilde\tau_L(b)} =
  \tau_{\overline{L}}(\overline{b})$. Since $\widetilde\tau_L$ is
  isotropic, we may find a nonzero $b\in\widetilde B_L$ such that
  $\widetilde\tau_L(b)\cdot b=0$. Multiplying $b$ on the right by a
  suitable power of $\xi$, we may assume $w(b)=0$, hence
  $\overline{b}\in B\otimes_F\overline{L}$ is defined and nonzero. We
  have
  \[
  \tau_{\overline{L}}(\overline{b})\cdot \overline{b} =
  \overline{\widetilde\tau_L(b)\cdot b} =0,
  \]
  hence $\tau_{\overline{L}}$ is isotropic. Note that $\overline{L}$
  is an odd-degree extension of $F$ since $L$ is an odd-degree
  extension of $F((x))$.
\end{proof}

In \cite[\S4]{MR1850658}, Parimala--Sridharan--Suresh give an example of a
central simple algebra $B$ with an anisotropic unitary involution that
becomes isotropic over an odd-degree extension $L$ of the field of
symmetric central elements. The algebra $B$ in this example has odd
exponent (and is split by $L$).


\def\cprime{$'$}

\end{document}